\documentclass[oneside,english,reqno]{amsart}
\usepackage[T1]{fontenc}
\usepackage[latin9]{inputenc}
\usepackage{babel}
\usepackage{refstyle}
\usepackage{mathrsfs}
\usepackage{enumitem}
\usepackage{amstext}
\usepackage{amsthm}
\usepackage{amssymb}
\usepackage[all]{xy}
\PassOptionsToPackage{normalem}{ulem}
\usepackage{ulem}
\usepackage[unicode=true,pdfusetitle,
 bookmarks=true,bookmarksnumbered=false,bookmarksopen=false,
 breaklinks=false,pdfborder={0 0 1},backref=false,colorlinks=false]
 {hyperref}
\hypersetup{
 colorlinks=true,citecolor=blue,linkcolor=blue,linktocpage=true}

\makeatletter

\AtBeginDocument{\providecommand\thmref[1]{\ref{thm:#1}}}
\AtBeginDocument{\providecommand\defref[1]{\ref{def:#1}}}
\AtBeginDocument{\providecommand\lemref[1]{\ref{lem:#1}}}
\RS@ifundefined{subsecref}
  {\newref{subsec}{name = \RSsectxt}}
  {}
\RS@ifundefined{thmref}
  {\def\RSthmtxt{theorem~}\newref{thm}{name = \RSthmtxt}}
  {}
\RS@ifundefined{lemref}
  {\def\RSlemtxt{lemma~}\newref{lem}{name = \RSlemtxt}}
  {}

\numberwithin{equation}{section}
\numberwithin{figure}{section}
\numberwithin{table}{section}
\theoremstyle{plain}
\newtheorem{thm}{\protect\theoremname}[section]
  \theoremstyle{definition}
  \newtheorem{defn}[thm]{\protect\definitionname}
  \theoremstyle{remark}
  \newtheorem{rem}[thm]{\protect\remarkname}
  \theoremstyle{remark}
  \newtheorem{notation}[thm]{\protect\notationname}
  \theoremstyle{definition}
  \newtheorem{example}[thm]{\protect\examplename}
  \theoremstyle{remark}
  \newtheorem*{note*}{\protect\notename}
  \theoremstyle{plain}
  \newtheorem{cor}[thm]{\protect\corollaryname}
  \theoremstyle{plain}
  \newtheorem{lem}[thm]{\protect\lemmaname}
  \theoremstyle{remark}
  \newtheorem*{acknowledgement*}{\protect\acknowledgementname}

\providecommand{\MR}[1]{}

\usepackage{bbm}
\allowdisplaybreaks
\usepackage{needspace}

\setlist[enumerate]{itemsep=5pt,topsep=3pt}
\setlist[enumerate,1]{label=\textup{(}\arabic*\textup{)},ref=\arabic*}
\setlist[enumerate,2]{label=\textup{(}\alph*\textup{)},ref=\theenumi \alph*}

\newref{lem}{refcmd={Lemma \ref{#1}}}
\newref{thm}{refcmd={Theorem \ref{#1}}}
\newref{cor}{refcmd={Corollary \ref{#1}}}
\newref{sec}{refcmd={Section \ref{#1}}}
\newref{sub}{refcmd={Section \ref{#1}}}
\newref{chap}{refcmd={Chapter \ref{#1}}}
\newref{prop}{refcmd={Proposition \ref{#1}}}
\newref{exa}{refcmd={Example \ref{#1}}}
\newref{tab}{refcmd={Table \ref{#1}}}
\newref{rem}{refcmd={Remark \ref{#1}}}
\newref{def}{refcmd={Definition \ref{#1}}}
\newref{fig}{refcmd={Figure \ref{#1}}}
\newref{claim}{refcmd={Claim \ref{#1}}}

\AtBeginDocument{
  
}

\makeatother

  \providecommand{\acknowledgementname}{Acknowledgement}
  \providecommand{\corollaryname}{Corollary}
  \providecommand{\definitionname}{Definition}
  \providecommand{\examplename}{Example}
  \providecommand{\lemmaname}{Lemma}
  \providecommand{\notationname}{Notation}
  \providecommand{\notename}{Note}
  \providecommand{\remarkname}{Remark}
\providecommand{\theoremname}{Theorem}

\begin{document}

\title{Positive definite kernels and boundary spaces}

\author{Palle Jorgensen and Feng Tian}

\address{(Palle E.T. Jorgensen) Department of Mathematics, The University
of Iowa, Iowa City, IA 52242-1419, U.S.A. }

\email{palle-jorgensen@uiowa.edu}

\urladdr{http://www.math.uiowa.edu/\textasciitilde{}jorgen/}

\address{(Feng Tian) Department of Mathematics, Hampton University, Hampton,
VA 23668, U.S.A.}

\email{feng.tian@hamptonu.edu}



\maketitle
\pagestyle{myheadings}
\markright{}
\begin{abstract}
We consider a kernel based harmonic analysis of \textquotedblleft boundary,\textquotedblright{}
and boundary representations. Our setting is general: certain classes
of positive definite kernels. Our theorems extend (and are motivated
by) results and notions from classical harmonic analysis on the disk.
Our positive definite kernels include those defined on infinite discrete
sets, for example sets of vertices in electrical networks, or discrete
sets which arise from sampling operations performed on positive definite
kernels in a continuous setting. 

Below we give a summary of main conclusions in the paper: Starting
with a given positive definite kernel $K$ we make precise generalized
boundaries for $K$. They are measure theoretic \textquotedblleft boundaries.\textquotedblright{}
Using the theory of Gaussian processes, we show that there is always
such a generalized boundary for any positive definite kernel.
\end{abstract}

\tableofcontents{}

\section{Introduction}

Our purpose is to make precise a variety of notions of \textquotedblleft boundary\textquotedblright{}
and boundary representation for general classes of positive definite
kernels. And to prove theorems which allow us to carry over results
and notions from classical harmonic analysis on the disk to this wider
context (see \cite{MR1655831,MR1667822,MR1655832}). We stress that
our positive definite kernels include those defined on infinite discrete
sets, for example sets of vertices in electrical networks, or discrete
sets which arise from sampling operations performed on positive definite
kernels in a continuous setting, and with the sampling then referring
to suitable discrete subsets. See, e.g., \cite{MR3285408,MR3541255,MR2884231}.

Below we give a summary of main conclusions in the paper: Starting
with a given positive definite kernel $K$ on $S\times S$, we introduce
generalized boundaries for the set $S$ that carries $K$. It is a
measure theoretic ``boundary\textquotedblright{} in the form of a
probability space, but it is not unique. The set of measure boundaries
will be denoted $\mathcal{M}\left(K\right)$. We show that there is
always such a generalized boundary probability space associated to
any positive definite kernel. For example, as an element in $\mathcal{M}\left(K\right)$,
we can take a \textquotedblleft measure\textquotedblright{} boundary
to be the Gaussian process having $K$ as its covariance kernel. This
exists by Kolmogorov's consistency theorem. 
\begin{defn}
\label{def:ps}By a \emph{probability space}, we mean a triple $\left(B,\mathcal{F},\mu\right)$
where:
\begin{itemize}
\item $B$ is a set,
\item $\mathcal{F}$ is a $\sigma$-algebra of subsets of $B$, and
\item $\mu$ is a probability measure defined on $\mathcal{F}$, i.e., $\mu\left(\emptyset\right)=0$,
$\mu\left(B\right)=1$, $\mu\left(F\right)\geq0$ $\forall F\in\mathcal{F}$,
and if $\left\{ F_{i}\right\} _{i\in\mathbb{N}}\subset\mathcal{F}$,
$F_{i}\cap F_{j}=\emptyset$, $i\neq j$ in $\mathbb{N}$, then $\mu\left(\cup_{i}F_{i}\right)=\sum_{i}\mu\left(F_{i}\right)$. 
\end{itemize}
\end{defn}

\paragraph*{Conclusions, a summary: }
\begin{enumerate}
\item For every positive definite kernel $K$, we define a \textquotedblleft measure
theoretic boundary space\textquotedblright{} $\mathcal{M}\left(K\right)$.
Set 
\begin{align*}
\mathcal{M}\left(K\right):= & \big\{\left(B,\mathcal{F},\mu\right)\;\mbox{a measure space which yields}\\
 & \;\mbox{a factorization for }K,\mbox{ see Definition \ref{def:Mk}}\big\}.
\end{align*}
This set $\mathcal{M}\left(K\right)$ generalizes other notions of
\textquotedblleft boundary\textquotedblright{} used in the literature
for networks, and for more general positive definite kernels, and
their associated reproducing kernel Hilbert spaces (RKHSs). See, e.g.,
\cite{MR3450534,MR3402823,MR3263485,MR3251728}.
\item For any positive definite kernel $K$, the corresponding $\mathcal{M}\left(K\right)$
is always non-empty. The natural Gaussian process path-space with
covariance kernel $K$, and Wiener measure $\mu$ is in $\mathcal{M}\left(K\right)$. 
\item \label{enu:sm3}Given $K$, let $\mathscr{H}\left(K\right)$ be the
associated RKHS. Then for every $\mu\in\mathcal{M}\left(K\right)$
there is a canonical isometry $W_{\mu}$ mapping $\mathscr{H}\left(K\right)$
into $L^{2}\left(\mu\right)$. For details, see \thmref{isom}.
\item \label{enu:sm4}The isometry $W_{\mu}$ in (\ref{enu:sm3}) generally
does not map onto $L^{2}\left(\mu\right)$. 

It does however for the $\frac{1}{4}$-Cantor example, i.e., the restriction
of Hausdorff measure of dimension $\frac{1}{2}$ to the standard $\frac{1}{4}$-Cantor
set. In this case, we have a positive definite kernel on $\mathbb{D}\times\mathbb{D}$,
where $\mathbb{D}$ is the unit disk in the complex plane; and we
can take the circle as boundary for $\mathbb{D}$. For $\mu$, we
take the corresponding $\frac{1}{4}$- Cantor measure. But in general,
for positive definite functions $K$, a \textquotedblleft measure
theoretic boundary space\textquotedblright{} is much \textquotedblleft bigger\textquotedblright{}
than probability spaces on the metric boundary for $K$.
\item Using the isometries from (\ref{enu:sm3}), we can turn $\mathcal{M}\left(K\right)$
into a partially ordered set; see \defref{po}. Then, using Zorn's
lemma, one shows that $\mathcal{M}\left(K\right)$ always contains
minimal elements. The minimal elements are not unique. 
\item And even if $\mu$ is chosen minimal in $\mathcal{M}\left(K\right)$,
the corresponding isometry $W_{\mu}$ still generally does not map
onto $L^{2}\left(\mu\right)$. A case in point: the Szegö kernel,
and $\mu=$ Lebesgue measure on a period interval.
\end{enumerate}
\begin{rem}
The Cantor examples in (\ref{enu:sm4}) are special cases of affine-selfsimilarity
limit (fractal) contractions. See, e.g., \cite{MR3531184,MR3418124,MR3326023,MR2945155,MR2817339,MR2559724,MR2440130,MR2240643}. 

The general role for the fractal dimension in these cases is as follows:
\[
\dim_{fractal}=\frac{\ln s}{\ln d}=\log_{d}\left(s\right),
\]
where $s=$ the number of translations in each iteration, and $d=$
the linear scale. For example, the middle-third Cantor fractal has
$\dim_{F}=\frac{\ln2}{\ln3}=\log_{3}\left(2\right)$. The Sierpinski-gasket
has $\dim_{F}=\frac{\ln3}{\ln2}<2$. For the Sierpinski construction
in $\mathbb{R}^{3}$, we have $\dim_{F}=\frac{\ln4}{\ln2}=2<3$. 
\end{rem}

\section{Generalized boundary spaces for positive definite kernels}
\begin{defn}
\label{def:pd}Let $S$ be any set. A function $K:S\times S\rightarrow\mathbb{C}$
is \emph{positive definite} iff (Def.) 
\begin{equation}
\sum_{i}\sum_{j}\overline{c_{i}}c_{j}K\left(s_{i},s_{j}\right)\geq0,\label{eq:n1}
\end{equation}
for all $\left\{ s_{i}\right\} _{i=1}^{n}\subset S$, and all $\left(c_{i}\right)_{i=1}^{n}\in\mathbb{C}^{n}$. 
\end{defn}
\begin{rem}
~
\begin{enumerate}[label=(\roman{enumi}),ref=\roman{enumi}]
\item \label{enu:ev1}Given a positive definite kernel $K$ on $S\times S$,
there is then an associated mapping $E_{S}:S\rightarrow\left\{ \mbox{Functions on }S\right\} $
given by 
\begin{equation}
E_{S}\left(t\right)=K\left(t,\cdot\right),\label{eq:ev1}
\end{equation}
where the dot ``$\cdot$'' in (\ref{eq:ev1}) indicates the independent
variable; so $S\ni s\longrightarrow K\left(t,s\right)\in\mathbb{C}$. 
\item \label{enu:ev2}We shall assume that $E_{S}$ is 1-1, i.e., if $s_{1},s_{2}\in S$,
and $k\left(s_{1},t\right)=k\left(s_{2},t\right)$, $\forall t\in S$,
then it follows that $s_{1}=s_{2}$. This is not a strong limiting
condition on $K$. 
\end{enumerate}
\end{rem}
\begin{notation}
We shall view the Cartesian product 
\begin{equation}
B_{S}:=\prod_{S}\mathbb{C}=\mathbb{C}^{S}\label{eq:ev2}
\end{equation}
as the set of all functions $S\rightarrow\mathbb{C}$. 

It follows from assumption (\ref{enu:ev2}) that $E_{S}:S\rightarrow B_{S}$
is an injection, i.e., with $E_{S}$, we may identity $S$ as a ``subset''
of $B_{S}$. 

For $v\in S$, set $\pi_{v}:B_{S}\longrightarrow\mathbb{C}$, 
\begin{equation}
\pi_{v}\left(x\right)=x\left(v\right),\quad\forall x\in B_{S};\label{eq:ev3}
\end{equation}
i.e., $\pi_{v}$ is the coordinate mapping at $v$. The topology on
$B_{S}$ shall be the product topology; and similarly the $\sigma$-algebra
$\mathcal{F}_{S}$ will be the the one generated by $\left\{ \pi_{v}\right\} _{v\in S}$,
i.e., generated by the family of subsets
\begin{equation}
\pi_{v}^{-1}\left(M\right),\;v\in S,\;\mbox{and }M\subset\mathbb{C}\;\mbox{a Borel set}.\label{eq:ev4}
\end{equation}
\end{notation}
\begin{defn}
\label{def:Mk}Fix a positive definite kernel $K:S\times S\rightarrow\mathbb{C}$.
Let $\mathcal{M}\left(K\right)$ be the set of all probability spaces
(see \defref{ps}), so that $\left(B,\mathcal{F},\mu\right)\in\mathcal{M}\left(K\right)$
iff (Def.) there exists an extension
\[
K^{B}:S\times B\longrightarrow\mathbb{C},\;\text{and}
\]
\begin{equation}
\int_{B}\overline{K^{B}\left(s_{1},b\right)}K^{B}\left(s_{2},b\right)d\mu\left(b\right)=K\left(s_{1},s_{2}\right),\label{eq:n2}
\end{equation}
for all $\left(s_{1},s_{2}\right)\in S\times S$.
\end{defn}
\begin{rem}
In Examples \ref{exa:D1}-\ref{exa:D2}, we discuss the case where
\begin{align*}
S & =\mathbb{D}=\left\{ z\in\mathbb{C}\mid\left|z\right|<1\right\} \\
B & =\partial\mathbb{D}=\left\{ z\in\mathbb{C}\mid\left|z\right|=1,\;\mbox{or }z=e^{ix},\;x\in(-\pi,\pi]\right\} ;
\end{align*}
but in the definition of $\mathcal{M}\left(K\right)$, we allow all
possible measure spaces $\left(B,\mathcal{F},\mu\right)$ as long
as the factorization (\ref{eq:n2}) holds.
\end{rem}

\paragraph*{Questions:}
\begin{enumerate}
\item Given (\ref{eq:n1}) what are the solutions $\left(B,\mathcal{F},\mu\right)$
to (\ref{eq:n2})? 
\item Are there extensions $K^{B}:S\times B\rightarrow\mathbb{C}$ such
that $B$ is a boundary with respect to the metric on $S$? That is,
\begin{equation}
dist_{K}\left(s_{1},s_{2}\right)=\left\Vert K_{s_{1}}-K_{s_{2}}\right\Vert _{\mathscr{H}}\label{eq:n3}
\end{equation}
and $\lim K^{B}\left(\cdot,b\right)=\lim_{i\rightarrow\infty}K\left(\cdot,s_{i}\right)$. 
\item Find the subsets $S_{0}\subset S$ such that the following sampling
property holds for all $f\in C\left(B\right)$ (or for a subspace
of $C\left(B\right)$):
\begin{equation}
f\left(b\right)=\sum_{s_{i}\in S_{0}}f\left(s_{i}\right)K^{B}\left(s_{i},b\right),\;\forall b\in B.\label{eq:n4}
\end{equation}
\end{enumerate}
\begin{example}[Shannon]
Let $BL$ be the space of band-limited functions on $\mathbb{R}$,
where 
\[
BL=\left\{ f\in L^{2}\left(\mathbb{R}\right)\mid\hat{f}\left(\xi\right)=0,\;\xi\in\mathbb{R}\backslash\left[-\tfrac{1}{2},\tfrac{1}{2}\right]\right\} .
\]
We have 
\begin{equation}
f\left(t\right)=\sum_{n\in\mathbb{Z}}f\left(n\right)\frac{\sin\pi\left(t-n\right)}{\pi\left(t-n\right)},\quad\forall t\in\mathbb{R},\;\forall f\in BL.\label{eq:n5}
\end{equation}
\end{example}
\begin{defn}
\label{def:GC}We say $\left(B,\mathcal{F},\mu\right)\in GC$, \emph{generalized
Carleson measures}, iff (Def.) there exists a constant $C_{\mu}$
such that 
\begin{equation}
\int_{B}|\widetilde{f}\left(b\right)|^{2}d\mu\left(b\right)\leq C_{\mu}\left\Vert f\right\Vert _{\mathscr{H}\left(K\right)}^{2},\;\forall f\in\mathscr{H}\left(K\right),\label{eq:n6}
\end{equation}
where $\widetilde{f}$ in (\ref{eq:n6}) is defined via the extension
\begin{equation}
\widetilde{f}\left(b\right):=\langle K_{b}^{B},f\rangle_{\mathscr{H}\left(K\right)},\quad b\in B,\:f\in\mathscr{H}\left(K\right).\label{eq:n7}
\end{equation}

Set $\left(GC\right)_{1}:=$ generalized Carleson measures with $C_{\mu}=1$. 
\end{defn}
\begin{note*}
The case $C_{\mu}=1$ is of special interest. For classical theory
on Carleson measures, we refer to \cite{MR739800,MR816238,MR2778503,MR2934601,MR3335506,MR3385785}.
\end{note*}
\begin{defn}
Let $\mathscr{H}_{i}$, $i=1,2$ be Hilbert spaces. We say that $\mathscr{H}_{1}$
is \emph{boundedly contained} in $\mathscr{H}_{2}$ iff (Def.) $\mathscr{H}_{1}\subset\mathscr{H}_{2}$
(as a subset), and if the inclusion map $\mathscr{H}_{1}\rightarrow\mathscr{H}_{2}$,
$h\mapsto h$, is bounded. That is, there exits $C<\infty$ such that
for all $h\in\mathscr{H}_{1}$, 
\begin{equation}
\left\Vert h\right\Vert _{\mathscr{H}_{2}}\leq C\left\Vert h\right\Vert _{\mathscr{H}_{1}.}\label{eq:bc1}
\end{equation}
 
\end{defn}
\begin{rem}
Note that if $\left(B,\mathcal{F},\mu\right)$ is a measure space,
$K:S\times S\rightarrow\mathbb{C}$ is a positive definite kernel,
then $\left(B,\mathcal{F},\mu\right)\in GC$ if and only if $\mathscr{H}\left(K\right)$
is boundedly contained in $L^{2}\left(B,\mathcal{F},\mu\right)$;
see (\ref{eq:n6}).

We stress that with the inclusion $\mathscr{H}\left(K\right)$ ``$\subset$''
$L^{2}\left(\mu\right)$ we can make the implicit identification $f\sim\widetilde{f}$
where 
\begin{equation}
\widetilde{f}\left(b\right)=\langle\widetilde{K}_{b},f\rangle_{\mathscr{H}\left(K\right)},\quad\forall f\in\mathscr{H}\left(K\right),\;b\in B;\label{eq:bc2}
\end{equation}
and (\ref{eq:bc2}) is to be understood for a.a. $b$ w.r.t. $\left(\mathcal{F},\mu\right)$.
\end{rem}
In \cite{2015arXiv150202549J}, we showed that for all positive definite
kernel $K\left(s,t\right)$, $\left(s,t\right)\in S\times S$, we
have $\mathcal{M}\left(K\right)\neq\emptyset$. Moreover, 
\begin{thm}
\label{thm:isom}Fix a positive definite kernel $K:S\times S\rightarrow\mathbb{C}$,
then 
\begin{equation}
\mathcal{M}\left(K\right)\subset\left(GC\right)_{1}.\label{eq:n8}
\end{equation}
If $\left(B,\mathcal{F},\mu\right)\in\mathcal{M}\left(K\right)$,
then the mapping
\begin{equation}
\mathscr{H}\left(K\right)\ni K(s,\underset{\substack{\uparrow\\
\mbox{on }S
}
}{\cdot})\longrightarrow K^{B}(s,\underset{\substack{\uparrow\\
\mbox{on }B
}
}{\cdot})\in L^{2}\left(B,\mu\right)\label{eq:m4}
\end{equation}
extends by linearity and closure to an isometry (see \defref{GC})
\[
W_{B}:\mathscr{H}\left(K\right)\longrightarrow L^{2}\left(B,\mu\right),\quad f\longrightarrow\widetilde{f}.
\]
However, $W_{B}$ is generally not onto $L^{2}\left(B,\mu\right)$. 

More specifically, we have 
\begin{equation}
\left\Vert \sum\nolimits _{j}c_{j}K\left(s_{j},\cdot\right)\right\Vert _{\mathscr{H}\left(K\right)}^{2}=\left\Vert \sum\nolimits _{j}c_{j}K^{B}\left(s_{j},\cdot\right)\right\Vert _{L^{2}\left(B,\mu\right)}^{2},\label{eq:m5}
\end{equation}
or equivalently, 
\begin{equation}
\sum_{j_{1}}\sum_{j_{2}}\overline{c_{j_{1}}}c_{j_{2}}K\left(s_{j_{1}},s_{j_{2}}\right)=\int_{B}\left|\sum\nolimits _{j}c_{j}K^{B}\left(s_{j},b\right)\right|^{2}d\mu\left(b\right)
\end{equation}
for all finite sums, where $\left\{ s_{j}\right\} $, $\left\{ c_{j}\right\} \subset\mathbb{C}^{n}$,
$\forall n\in\mathbb{N}$. 
\end{thm}
\begin{proof}
Suppose $\left(B,\mathcal{F},\mu\right)\in\mathcal{M}\left(K\right)$,
i.e., assume $\left(B,\mathcal{F},\mu\right)$ is a measure space
such that (\ref{eq:n2}) holds. Set $K^{B}=\widetilde{K}$, refer
to the extension $\widetilde{K}:S\times B\rightarrow\mathbb{C}$ introduced
in (\ref{eq:n3}).

We claim that then (\ref{eq:n6}) holds for all $f\in\mathscr{H}\left(K\right)$.
Here $\widetilde{f}$ is defined via $\widetilde{K}$; see (\ref{eq:n7}):
\[
\widetilde{f}\left(b\right):=\langle\widetilde{K}_{b},f\rangle_{\mathscr{H}\left(K\right)},\;\forall f\in\mathscr{H}\left(K\right),\;\forall b\in B..
\]

Claim: $f\mapsto\widetilde{f}$ is isometric from $\mathscr{H}\left(K\right)$
into $L^{2}\left(\mu\right)$, i.e., 
\begin{equation}
\Vert\widetilde{f}\Vert_{L^{2}\left(B,\mu\right)}=\left\Vert f\right\Vert _{\mathscr{H}\left(K\right)},\;\forall f\in\mathscr{H}\left(K\right).\label{eq:n10}
\end{equation}
\emph{Proof of }(\ref{eq:n10})\emph{. }It is enough to consider the
case where $f=\sum_{i}c_{i}K_{s_{i}}$ (finite sum), see (\ref{eq:n1});
so that $\widetilde{f}=\sum_{i}c_{i}\widetilde{K}_{s_{i}}$ on $B$,
and 
\begin{align*}
\Vert\widetilde{f}\Vert_{L^{2}\left(B,\mu\right)}^{2} & =\sum_{i}\sum_{j}\overline{c_{i}}c_{j}\langle\widetilde{K}_{s_{i}}\widetilde{K}_{s_{j}}\rangle_{L^{2}\left(B,\mu\right)}\\
 & =\sum_{i}\sum_{j}\overline{c_{i}}c_{j}\int_{B}\overline{\widetilde{K}_{s_{i}}\left(b\right)}\widetilde{K}_{s_{j}}\left(b\right)d\mu\left(b\right)\\
 & =\sum_{i}\sum_{j}\overline{c_{i}}c_{j}K\left(s_{i},s_{j}\right)\quad\left(\text{see }\left(\ref{eq:n2}\right),\;\text{use }\mu\in\mathcal{M}\left(K\right)\right)\\
 & =\left\Vert f\right\Vert _{\mathscr{H}\left(K\right)}^{2},\quad\text{by }\left(\ref{eq:n1}\right)\;\text{and the defn. of }\mathscr{H}\left(K\right).
\end{align*}
 
\end{proof}
\begin{cor}
Suppose $\mathscr{H}\left(K\right)\ni f\xrightarrow{\;W_{B}\;}\widetilde{f}\in L^{2}\left(B,\mu\right)$
is bounded, i.e., that $\mu$ is a Carleson measure, then the adjoint
operator
\[
W_{B}^{*}:L^{2}\left(B,\mu\right)\longrightarrow\mathscr{H}\left(K\right)
\]
 is given by 
\begin{equation}
W_{B}^{*}\left(F\right)\left(s\right)=\int_{B}\overline{\widetilde{K}\left(s,b\right)}F\left(b\right)d\mu\left(b\right),\;\forall F\in L^{2}\left(B,\mu\right).\label{eq:n11}
\end{equation}
\end{cor}
\begin{proof}
For all $F\in L^{2}\left(B,\mu\right)$, and all $s\in S$, we have
\begin{align*}
\left\langle K_{s},W_{B}^{*}F\right\rangle _{\mathscr{H}\left(K\right)} & =\left(W_{B}^{*}F\right)\left(s\right)\quad\left(\text{reprod prop., and }W_{B}^{*}F\in\mathscr{H}\left(K\right)\right)\\
 & =\left\langle W_{B}K_{s},F\right\rangle _{L^{2}\left(\mu\right)}\quad\left(\text{by duality}\right)\\
 & =\int_{B}\overline{\widetilde{K}\left(s,b\right)}F\left(b\right)d\mu\left(b\right)
\end{align*}
which is the desired conclusion (\ref{eq:n11}).
\end{proof}
We now turn to the Gaussian measure boundary:
\begin{cor}
Suppose $K:S\times S\rightarrow\mathbb{C}$ is a given positive definite
kernel, and that there is a measure space $\left(\mathcal{F},\mu\right)$
where $\mathcal{F}$ is a $\sigma$-algebra of subsets of $S$, such
that the RKHS $\mathscr{H}\left(K\right)$ satisfies $\mathscr{H}\left(K\right)\subset L^{2}\left(S,\mathcal{F},\mu\right)$
(isometric inclusion), then $\left(S,\mathcal{F},\mu\right)\in\mathcal{M}\left(K\right)$
iff
\begin{equation}
K\left(s,t\right)=\int_{S}\overline{K\left(s,x\right)}K\left(t,x\right)d\mu\left(x\right),\quad\forall\left(s,t\right)\in S\times S.\label{eq:pb1}
\end{equation}
 
\end{cor}
\begin{example}
The condition in (\ref{eq:pb1}) is satisfied for Bargmann's Hilbert
space $\mathscr{H}$ of entire analytic functions on $\mathbb{C}$
(see \cite{MR924682,MR0201959}) subject to 
\begin{align}
\left\Vert f\right\Vert _{\mathscr{H}}^{2} & =\frac{1}{2\pi}\iint_{\mathbb{R}^{2}}\left|f\left(x+i\,y\right)\right|^{2}e^{-\frac{x^{2}+y^{2}}{2}}dx\,dy\label{eq:pb2}\\
 & =\frac{1}{2\pi}\int_{\mathbb{C}}\left|f\left(z\right)\right|^{2}e^{-\frac{\left|z\right|^{2}}{2}}dx\,dy<\infty.\nonumber 
\end{align}

The following kernel (Bargmann's kernel) is positive definite on $\mathbb{C}\times\mathbb{C}$:
\begin{equation}
K\left(z,w\right)=\exp\left(\frac{\overline{z}w}{2}-\frac{\left|z\right|^{2}+\left|w\right|^{2}}{4}\right).\label{eq:pb3}
\end{equation}
It is known that $K$ in (\ref{eq:pb3}) satisfies (\ref{eq:pb1})
with respect to the measure $\mu$ on $\mathbb{C}$, given by 
\begin{equation}
d\mu\left(z\right)=dA\left(z\right)=\frac{1}{2\pi}dx\,dy.\label{eq:pb4}
\end{equation}
\end{example}
\begin{thm}
Let $\left(K,S\right)$ be a positive definite kernel (\defref{pd})
such that the associated mapping $E_{S}:S\rightarrow B_{S}$ is 1-1
(see (\ref{eq:ev1})). 

Then there is a probability space $\left(B_{S},\mathcal{F}_{S},\mu_{S}\right)$
which satisfies the condition (\ref{eq:n1}) in \defref{pd}. 
\end{thm}
\begin{proof}
This argument is essentially the Kolmogorov inductive limit construction.
For every $n\in\mathbb{N}$, $\forall\left\{ s_{1},\cdots,s_{n}\right\} \subset S$,
we associate a measure $\mu^{\left\{ s_{1},\cdots,s_{n}\right\} }$
on $B_{S}$ as follows: 

Let $\mu^{\left\{ s_{1},\cdots,s_{n}\right\} }$ be the measure on
$B_{S}$ which has $\left(\pi_{s_{1}},\cdots,\pi_{s_{n}}\right)$
as an $n$ vector valued random variable with Gaussian the specific
distribution: mean zero, and joint covariance matrix $\left\{ K\left(s_{i},s_{j}\right)\right\} _{i,j=1}^{n}$.
By a standard argument, one checks that then $\mu^{\left\{ s_{1},\cdots,s_{n}\right\} }$
is a consistent system of measures on $B_{S}$; and (by Kolmogorov)
that there is a unique probability measure $\mu_{S}$ on the measure
space $\left(B_{S},\mathcal{F}_{S}\right)$ such that, for all $\left(s_{1},\cdots,s_{n}\right)$,
the marginal distribution of $\mu_{S}$ coincides with $\mu^{\left\{ s_{1},\cdots,s_{n}\right\} }$. 
\end{proof}
\begin{example}[$W_{B}$ is onto]
\label{exa:D1}Let 
\begin{align}
V & =\mathbb{D}=\left\{ z\in\mathbb{C}\mid\left|z\right|<1\right\} \\
B & =\partial\mathbb{D}=\left\{ z\in\mathbb{C}\mid\left|z\right|=1,\;\mbox{or }z=e^{ix},\;x\in(-\pi,\pi]\right\} .
\end{align}
Set 
\[
K\left(z,w\right)=\prod_{l=0}^{\infty}\left(1+\left(\overline{z}w\right)^{4^{l}}\right),\quad\left(z,w\right)\in\mathbb{D}\times\mathbb{D},
\]
and 
\[
K^{B}\left(z,x\right)=\prod_{l=0}^{\infty}\left(1+\left(\overline{z}e^{i2\pi x}\right)^{4^{l}}\right),\quad\left(z,x\right)\in\mathbb{D}\times B.
\]
Then (\ref{eq:n2}) holds for the case when $\mu_{\frac{1}{4}}$ =
the $\frac{1}{4}$-Cantor measure on $B$; see \cite{MR2240643}.
\end{example}
\begin{proof}
(Sketch) Set 
\begin{align}
\Lambda_{4} & =\left\{ \sum_{i=0}^{n}b_{i}4^{i}\mid b_{i}\in\left\{ 0,1\right\} ,\:n\in\mathbb{N}\right\} \\
 & =\left\{ 0,1,4,5,16,17,20,21,64,65,\cdots\right\} \nonumber 
\end{align}
then 
\begin{equation}
\prod_{l=0}^{\infty}\left(1+t^{4^{l}}\right)=\sum_{\lambda\in\Lambda}t^{\lambda},\quad\left|t\right|<1.
\end{equation}
The desired conclusion 
\begin{equation}
K\left(z,w\right)=\int_{C_{\frac{1}{4}}}\overline{K_{C_{\frac{1}{4}}}\left(z,x\right)}K_{C_{\frac{1}{4}}}\left(w,x\right)d\mu_{\frac{1}{4}}\left(x\right)
\end{equation}
follows from the fact that $\left\{ e_{\lambda}\mid\lambda\in\Lambda_{4}\right\} $
is an ONB in $L^{2}\left(C_{1/4},\mu_{1/4}\right)$ by \cite{MR1655831}. 
\end{proof}

\section{Boundary theory}

We now turn to the details regarding boundary theory. To connect it
to the classical theory of kernel spaces of analytic functions on
the disk, we begin with an example, and we then turn to the case of
the most general positive definite kernels; but not necessarily restricting
the domain of the kernels to be considered.
\begin{example}[$W_{B}$ is not onto]
\label{exa:D2} Let 
\begin{align}
K\left(z,w\right) & =\frac{1}{1-\overline{z}w},\;(\mbox{Szegö kernel})\nonumber \\
K^{B}\left(z,x\right) & =\frac{1}{1-\overline{z}e^{i2\pi x}},\;\mbox{and}\\
\mu & =\mbox{restriction of Lebesgue measure to }\left[0,1\right].\nonumber 
\end{align}
Let $\mathbb{H}_{2}$ be the Hardy space on $\mathbb{D}$. Then 
\[
W_{B}:\mathbb{H}_{2}\longrightarrow L^{2}\left(\left[0,1\right],\mu_{Leb}\right)
\]
is isometric, but not onto. Indeed, 
\[
W_{B}\left(\mathbb{H}_{2}\right)=\overline{span}^{L^{2}\left(0,1\right)}\left\{ e_{n}\left(x\right)\mid n\in\mathbb{N}_{0}=\left\{ 0\right\} \cup\mathbb{N}\right\} .
\]
\end{example}
Returning to the general case, we show below that there is always
a minimal element in $\mathcal{M}\left(K\right)$; see \defref{Mk}. 
\begin{defn}
\label{def:po}Suppose $\left(B_{i},\mathcal{F}_{i},\mu_{i}\right)\in\mathcal{M}\left(K\right)$,
$i=1,2$. We say that 
\begin{equation}
\left(B_{1},\mathcal{F}_{1},\mu_{1}\right)\leq\left(B_{2},\mathcal{F}_{2},\mu_{2}\right)\label{eq:mo0}
\end{equation}
if $\exists\varphi:B_{2}\longrightarrow B_{1}$, s.t. 
\begin{align}
\mu_{2}\circ\varphi^{-1} & =\mu_{1},\;\mbox{and}\label{eq:mo1}\\
\varphi^{-1}\left(\mathcal{F}_{1}\right) & =\mathcal{F}_{2}.\label{eq:mo2}
\end{align}
\end{defn}
\begin{lem}
\label{lem:me}$\mathcal{M}\left(K\right)$ has minimal elements.
\end{lem}
\begin{proof}
If (\ref{eq:mo1})-(\ref{eq:mo2}) hold, then 
\[
L^{2}\left(B_{1},\mu_{1}\right)\ni f\xrightarrow{\quad W_{21}}f\circ\varphi\in L^{2}\left(B_{2},\mu_{2}\right)
\]
is isometric, i.e., 
\begin{equation}
\int_{B_{2}}\underset{W_{21}f}{|\underbrace{f\circ\varphi}|^{2}}d\mu_{2}=\int_{B_{1}}\left|f\right|^{2}d\mu_{1},
\end{equation}
and
\begin{equation}
W_{B_{2}}=W_{21}W_{B_{1}}\;\mbox{on }\mathscr{H}\left(K\right),
\end{equation}
i.e., the diagram commutes:
\[
\xymatrix{ &  & L^{2}\left(B_{1},\mu_{1}\right)\ar[d]^{W_{21}}\\
\mathscr{H}\left(K\right)\ar@/^{1.3pc}/[rru]^{W_{B_{1}}}\ar[rr]_{W_{B_{2}}} &  & L^{2}\left(B_{2},\mu_{2}\right)
}
\]

We can then use Zorn's lemma to prove that $\forall K$, $\mathcal{M}\left(K\right)$
has minimal elements $\left(B,\mathcal{F},\mu\right)$. But even if
$\left(B,\mathcal{F},\mu\right)$ is minimal, $W_{B}:\mathscr{H}\left(K\right)\rightarrow L^{2}\left(\mu\right)$
may \emph{not} be onto.
\end{proof}
In the next result, we shall refer to the partial order ``$\leq$''
from (\ref{eq:mo0}) when considering minimal elements in $\mathcal{M}\left(K\right)$.
And, in referring to $\mathcal{M}\left(K\right)$, we have in mind
a fixed positive definite function $K:S\times S\rightarrow\mathbb{C}$,
specified at the outset; see Definitions \ref{def:pd} and \ref{def:Mk}.
\begin{thm}
Let $\left(K,S\right)$ be a fixed positive definite kernel, and let
$\mathcal{M}\left(K\right)$ be the corresponding boundary space from
\defref{Mk}. 

Then, for every $\left(X,\lambda\right)\in\mathcal{M}\left(K\right)$,
there is a $\left(M,\nu\right)\in\mathcal{M}\left(K\right)$ such
that 
\begin{equation}
\left(M,\nu\right)\leq\left(X,\lambda\right),\label{eq:pa1}
\end{equation}
and $\left(M,\nu\right)$ is \uline{minimal} in the following sense:
Suppose $\left(B,\mu\right)\in\mathcal{M}\left(K\right)$ and 
\begin{equation}
\left(B,\mu\right)\leq\left(M,\nu\right),\label{eq:pa2}
\end{equation}
then it follows that $\left(B,\mu\right)\simeq\left(M,\nu\right)$,
i.e., we also have $\left(M,\nu\right)\leq\left(B,\mu\right)$. 
\end{thm}
\begin{proof}
We shall use Zorn's lemma, and the argument from \lemref{me}.

Let $\mathcal{L}=\left\{ \left(B,\mu\right)\right\} $ be a linearly
ordered subset of $\mathcal{M}\left(K\right)$ s.t. 
\begin{equation}
\left(B,\mu\right)\leq\left(X,\lambda\right),\quad\forall\left(B,\mu\right)\in\mathcal{L};\label{eq:pa3}
\end{equation}
and such that, for every pair $\left(B_{i},\mu_{i}\right)$, $i=1,2$,
in $\mathcal{L}$, one of the following two cases must hold:
\begin{equation}
\left(B_{1},\mu_{1}\right)\leq\left(B_{2},\mu_{2}\right),\;\mbox{or }\left(B_{2},\mu_{2}\right)\leq\left(B_{1},\mu_{1}\right).\label{eq:pa4}
\end{equation}
To apply Zorn's lemma, we must show that there is a $\left(B_{\mathcal{L}},\mu_{\mathcal{L}}\right)\in\mathcal{M}\left(K\right)$
such that 
\begin{equation}
\left(B_{\mathcal{L}},\mu_{\mathcal{L}}\right)\leq\left(B,\mu\right),\quad\forall\left(B,\mu\right)\in\mathcal{L}.\label{eq:pa5}
\end{equation}

Now, using (\ref{eq:pa3})-(\ref{eq:pa4}), we conclude that the measure
spaces $\left\{ \left(B,\mu\right)\right\} _{\mathcal{L}}$ have an
inductive limit, i.e., the existence of:
\begin{equation}
\mu_{\mathcal{L}}:=\underset{B\xrightarrow[\;\mathcal{L}\;]{}B_{\mathcal{L}}}{\mbox{ind limit }}\mu_{B}.\label{eq:pa6}
\end{equation}
In other words, we may apply Kolmogorov's consistency to the family
$\mathcal{L}$ of measure spaces in order to justify the inductive
limit construction in (\ref{eq:pa6}).

We have proved that every linearly ordered subset $\mathcal{L}$ (as
specified) has a ``lower bound'' in the sense of (\ref{eq:pa5}).
Hence Zorn's lemma applies, and the desired conclusion follows, i.e.,
there is a pair $\left(M,\nu\right)\in\mathcal{M}\left(K\right)$
which satisfies the condition (\ref{eq:pa2}) from the theorem. 
\end{proof}
\begin{acknowledgement*}
The co-authors thank the following colleagues for helpful and enlightening
discussions: Professors Sergii Bezuglyi, Ilwoo Cho, Paul Muhly, Myung-Sin
Song, Wayne Polyzou, and members in the Math Physics seminar at The
University of Iowa.
\end{acknowledgement*}
\bibliographystyle{amsalpha}
\bibliography{ref}

\end{document}